\documentclass[11pt]{article}
\usepackage{a4wide}

\usepackage[a4paper, margin=2.3cm]{geometry}
\usepackage{amsmath,amssymb,amsthm, comment, enumitem, array}
\usepackage{color}
\usepackage{parskip}
\usepackage{hyperref}
\usepackage{parskip}
\usepackage{setspace}
\usepackage{graphicx}
\usepackage{mathtools}
\usepackage[thinc]{esdiff}
\usepackage[euler]{textgreek}
\usepackage{subfig}
\usepackage{xcolor}

\usepackage{scalerel,stackengine}
\stackMath
\newcommand\reallywidehat[1]{%
\savestack{\tmpbox}{\stretchto{%
  \scaleto{%
    \scalerel*[\widthof{\ensuremath{#1}}]{\kern-.6pt\bigwedge\kern-.6pt}%
    {\rule[-\textheight/2]{1ex}{\textheight}}
  }{\textheight}%
}{0.5ex}}%
\stackon[1pt]{#1}{\tmpbox}%
}

\newtheorem{theorem}{Theorem}[section]

\newtheorem{remark}{Remark}[section]
\newtheorem{corollary}[theorem]{Corollary}
\newtheorem{lemma}[theorem]{Lemma}

\newtheorem{definition}[theorem]{Definition}

\newtheorem*{definition*}{Definition}
\newtheorem{example}[theorem]{Example}

\hypersetup{
    colorlinks=true,
    linkcolor = blue,
    citecolor=magenta,
    urlcolor=cyan,
}

\usepackage[capitalise, nameinlink]{cleveref}
\crefname{equation}{}{}
\crefname{figure}{{\sc Figure}}{{\sc Figure}}
\crefname{subsection}{Subsection}{Subsections}

\def\bC{\mathbb{C}}
\def\bE{\mathbb{E}}
\def\cN{\mathcal{N}}

\def\bF{\mathbb{F}}

\def\cM{\mathcal{M}}

\def\cM{\mathcal{M}}

\def\fM{\mathfrak{M}}
\def\fE{\mathfrak{E}}
\def\cX{\mathcal{X}}

\def\mod{\textsf{mod }}

\def\Ker{\textsf{Ker}}
\def\HS{\textsf{HS}}
\def\i{\mathfrak{i}}

\def\<{\langle}
\def\>{\rangle}

\begin{document}

\begin{center}
\LARGE New-type Quasirandom Groups and Applications 
\end{center}

\begin{center}
Thang Pham\footnote{University of Science, Vietnam National University, Hanoi. Email: phamanhthang.vnu@gmail.com} and Boqing Xue\footnote{Institute of Mathematical Sciences, ShanghaiTech University. Email: xuebq@shagnhaitech.edu.cn}
\end{center}

\begin{abstract}
This paper aims to introduce a more general definition of quasirandom groups and  generalize several well-known results in the literature in this new setting. More precisely, let $G$ be a semi-direct product of groups and $X\subseteq G$, we provide conditions such that one can find tuples $(x_0, \ldots, x_k)\in X^{k+1}$ satisfying $x_1x_2\ldots x_k=x_0$ or conditions to guarantee that the product set $XX$ grows exponentially. In a special case of the group of rigid-motions in the plane over an arbitrary finite field, our results offer a reasonably complete description of structures of this group. 
\end{abstract}
\bigskip



\section{Introduction}
Let $(G,\cdot)$ be a finite group and $D$ be a positive integer. The group $G$ is called \textit{$D$-quasirandom} if all non-trivial representations of $G$ are of degree at least $D$. This notion was introduced by Gowers \cite{W} in his solution of the following question due to Babai and S\'{o}s \cite{BS}: Does there exist a constant $c>0$ such that every finite group $G$ has a product-free subset of size at least $c|G|$?

Gowers proved the following theorem.
\begin{theorem} \label{thm_Gowers}
If $G$ is $D$-quasirandom, then
\begin{equation}\label{one}
\left\vert \#\{(x,y,xy)\in X\times Y\times Z\}-\frac{|X||Y||Z|}{|G|}\right\vert \ll \sqrt{\frac{|G||X||Y||Z|}{D}}
\end{equation}
for every $X,Y,Z\subseteq G$. 
\end{theorem}
This means that when $X, Y$ and $Z$ are large enough, then the number of triples $(x, y, xy)\in X\times Y\times Z$ is close to the expected value. As a direct application, if $X\subseteq G$ is a product-free subset of $G$, then we have $|X|\ll |D|^{-1/3}|G|$. Gowers also asked whether a similar result holds for three-term progressions, i.e. triples of the form $(x, xy, xy^2)$. This question was solved by Tao \cite{Terry} for $SL_d(\mathbb{F}_q)$, Peluse \cite{Pau} for non-abelian finite simple groups, and by Bhangale, Harsha and Roy \cite{BHR} for all finite quasirandom groups. 

Another interesting application of the estimate (\ref{one}) is on the growth of product of sets in quasirandom groups. More precisely, for $X, Y\subseteq G$, we set $Z=XY$, then 
\begin{equation} \label{eq_prodgrowth}
|XY|\gg \min \left\lbrace |G|, \,\frac{D}{|G|} |X||Y|\right\rbrace.
\end{equation}
We note that a more general statement of this result in the form of $||1_X\ast 1_Y||_{2}^2$ was studied by Babai, Nikolov, and Pyber \cite{BNP}. 

In the setting of the group $SL_2(\mathbb{F}_q)$, it is well--known that $D=(q-1)/2$ (see \cite{DSV}, page 102). The estimate \eqref{eq_prodgrowth} gives $|XX|\gg \min \left\lbrace q^3, q^{-2}|X|^2 \right\rbrace$, so any subset $X$ of $SL_2(\bF_q)$ with $|X|=q^{\alpha}$ $(\alpha>2)$ has exponentially product growth. 

It is worth noting that computing the product of elements from a group is a fundamental problem in theoretical computer science. Gowers and Viola \cite{GV} studied mixing in several non-quasirandom groups, and obtained results on communication complexity. One of the models considered in their paper is the affine group $\textsf{Aff}(\bF_q)$ over $\bF_q$, which is a semi-product group. Since this group has large subgroups, the estimates \eqref{one} and \eqref{eq_prodgrowth} do not hold anymore. 

The main purpose of this paper is to extend the estimate (\ref{one}) in the setting of semi-direct product of groups, i.e. groups $G$ of the form $N\rtimes_\varphi H$. If $H=\{1\}$ (the trivial multiplicative group), then our result recovers the estimate (\ref{one}). Our initial motivation of studying this topic comes from the following question: given a set $X$ of rigid-motions in the plane over an arbitrary finite field $\mathbb{F}_q$, under what conditions on $X$ does the set $XX$ grow exponentially? 

We start with a simple observation. Consider the group $G_0=\bF_q^2\rtimes_\varphi SO_2(\bF_q)$ of rigid motions with $|G_0| = q^2(q\pm 1)$. (See more details in section \ref{sec_rm} and Example \ref{ex1}.) Let $\gamma$ be a generator of the cyclic group $SO_2(\bF_q)$. Assume that $|SO_2(\bF_q)|=q\pm 1 =kl$, and let 
\[
\cX=\{(z,\gamma^{kj}):\, z\in \mathbb{F}_q^2,\, 0\leq j\leq l-1\}.
\]
Then $\cX$ is a subgroup of $G_0$ and $\cX\cX=\cX$. Here $|\cX|/|G_0|=1/k$. This infers that, it is possible to choose an arbitrarily large subset $X\subseteq G_0$ such that the product $XX$ does not grow in general. In other words, a condition on the size of $X$ is not enough to guarantee the expanding property of $XX$. 

The second observation we want to mention here is that, one can apply the estimate (\ref{eq_prodgrowth}) on $G_0$ to obtain that 
\[|\cX|=|\cX\cX|\gg \min \left\lbrace q^3, ~\frac{D}{q^3}|\cX|^2\right\rbrace \gg \frac{D}{k}|\cX|.\]
The above example suggests that the number $D$ for the group of rigid motions should be at most a constant which does not depend on $q$. This is true and will be confirmed in Theorem \ref{thm_repn_dim_H0} below. 

Putting these two observations together, we realize that in order to understand structures of the product set $XX$, a deeper studying is needed. 

In this paper, our results will be stated and proved in the setting of semi-direct product groups. The group of rigid-motions is a special case. The main tool we will use is non-abelian Fourier analysis.
\subsection{Main results on semi-direct product groups}
Let $(N,\cdot)$ and $(H,\cdot)$ be two groups with identities $1_N$ and $1_H$, respectively. Assume that $\varphi:\, H\rightarrow Aut(N)$ is a group homomorphism. For simplicity, we denote $\varphi_h=\varphi(h)$ for $h\in H$. The semi-direct product $(G,\cdot)=(N,\cdot)\rtimes_\varphi (H,\cdot)$ is a group of order $|N||H|$. More explicitly, one has
\[
G=\{(z,h):\,  z\in N, h\in H\}
\]
with the group law given by 
\[
(z_1,h_1)(z_2,h_2) = (z_1\,\varphi_{h_1}(z_2),\, h_1h_2).
\]
Through this paper, for each $g\in G$, we write $g:=(\ddot{g},\dot{g})$ or $g=(g^{\cdot\cdot}, g^{\cdot})$.


The subgroup $\widetilde{N}:=\{(z,1_H):\, z\in N\}$ is a normal subgroup of $G$, which gives
\[
G/\widetilde{N} \simeq H.
\] 
Denote by $\pi:\, G\rightarrow G/\widetilde{N}$ the quotient map. If $\rho:\, G\rightarrow GL(V)$ is a complex representation of $G$ such that $\rho|_{\widetilde{N}}$ is the identity, then $\Ker(\rho)\supseteq \widetilde{N}$. As a result,  there is a unique homomorphism $\dot{\rho}:\, G/\widetilde{N}\rightarrow GL(V)$ such that $\rho = \dot{\rho}\pi$. Here, we can view $\dot{\rho}$ as a representation of $H$. On the other hand, for any representation $\dot{\rho}$ of $H \simeq G/\widetilde{N}$, the homomorphism $\rho=\dot{\rho}\pi$ is a representation of $G$.

Now all the representations of $G$ can be divided into the following two types: 

\textsf{\bf Type I:} a representation $\rho$ that is lifted from some representation $\dot{\rho}$ of $H$;

\textsf{\bf Type II:} a representation $\rho$ such that $\rho((z,1_H))$ is not the identity for some $z\in N$.



We are interested in the situation that type II representations have large degree. 

\begin{definition}
Let $D$ be a positive number. We say that $G= N\rtimes_\varphi H$ is $D$-quasirandom on $N$ if the degree of any representation of $G$ of type II is at least $D$.   
\end{definition}

We note that if a group $G'$ is $D$-quasirandom in Gowers's definition, then it is $D$-quasirandom in this definition by viewing $G' \simeq G'\rtimes \{1\}$. 
However, the inverse is not true, the group of rigid-motions (with Theorem \ref{thm_repn_dim_H0} below) is an example. 

Our first result states as follows. 
\begin{theorem} \label{thm_mixing}
Let $G=N\rtimes_\varphi H$ be defined as previous, which is $D$-quasirandom on $N$. For any $k\geq 2$ and $X_0,X_1,\ldots,X_k\subseteq G$, denote
\[
\cM_k = \#\{(x_0,x_1,\ldots,x_k)\in X_0\times X_1\times\ldots\times X_k:\, x_1x_2\cdots x_k=x_0\},
\]
and 
\[
\dot{\cM_k}= \#\{(x_0,x_1,\ldots,x_k)\in X_0\times X_1\times\ldots\times X_k:\, \dot{x_1}\dot{x_2}\cdots \dot{x_k}=\dot{x_0}\}.
\]
Then 
\[
\left|\cM_k -  \frac{1}{|N|}\dot{\cM_k}\right|\leq \sqrt{\frac{|G|^{k-1}|{X_0}||{X_1}|\ldots |{X_k}|}{D^{k-1}}}.
\]
\end{theorem}

Assume that $G'$ is $D$-quasirandom. Taking $k=2$ and $G=G'\rtimes \{1\}$ in the above theorem, one has $\dot{\cM_2}=|X_0||X_1||X_2|$. Then Theorem \ref{thm_Gowers} follows immediately from Theorem \ref{thm_mixing}.

For any subset $X\subseteq G$, denote by $1_X$ the characteristic function of $X$, i.e.,
\[
1_X(x) = 
\begin{cases}
1,\quad &\text{if }x\in X,\\
0,\quad &\text{if }x\notin X.
\end{cases}
\]
The next theorem shows upper and lower bounds for convolutions. 

\begin{theorem} \label{thm_l_2}
Let $G=N\rtimes_\varphi H$ be defined as previous, which is $D$-quasirandom on $N$. For any integer $k\geq 2$ and subsets $X_1,X_2,\ldots, X_k$ of $G$, we have 
\[
\frac{|H|\dot{\cN_k}}{|G|^{2k}}\leq \|1_{X_1}\ast 1_{X_2}\ast \ldots \ast 1_{X_k}\|_2^2 \leq \frac{|H|\dot{\cN_k}}{|G|^{2k}}+\frac{|X_1||X_2|\ldots |X_k|}{D^{k-1}|G|^k},
\]
where
\[
\dot{\cN_k} = \#\{(x_1,x_2,\ldots,x_k,y_1,y_2,\ldots,y_k)\in (X_1\times X_2\times\ldots\times X_k)^2:\, \dot{x_1}\dot{x_2}\cdots\dot{x_k}=\dot{y_1}\dot{y_2}\cdots\dot{y_k}\}.
\]
\end{theorem}

For any $X,Y\subseteq G$, the energies $E(X, Y)$ and $\dot{E}(X,Y)$ are defined by
\[
E(X,Y):=\#\{(x_1,x_2,y_1,y_2)\in X^2\times Y^2:\, x_1y_1=x_2y_2\},
\] 
and
\[
\dot{E}(X,Y) := \#\{(x_1,x_2,y_1,y_2)\in X^2\times Y^2:\, \dot{x_1}\dot{y_1}=\dot{x_2}\dot{y_2}\}.
\]
The next corollary follows directly from Theorem \ref{thm_l_2}.

\begin{corollary}
Let $G=N\rtimes_\varphi H$ be defined as previous, which is $D$-quasirandom on $N$. For any $X,Y\subseteq G$, we have 
\[
\frac{|H|}{|G|} \dot{E}(X,Y)\leq E(X,Y)\leq \frac{|H|}{|G|}\dot{E}(X,Y)+ \frac{|G|}{D}|X||Y|.
\]
\end{corollary}

The next theorem provides an estimate of possible product growth.

\begin{theorem} \label{thm_lowbound_XY}
Let $G=N\rtimes_\varphi H$ and $\dot{\cN_k}$ be defined as previous. Suppose that $G$ is $D$-quasirandom on $N$. For an integer $k\geq 2$ and subsets $X_1,X_2,\ldots, X_k$ of $G$, we have 
\[
|X_1X_2\cdots X_k| \geq  \frac{1}{2}\min \left\{\, \frac{|N|}{\dot{\cN_k}}|X_1|^2|X_2|^2\cdots |X_k|^2,\,\,\frac{D^{k-1}}{|G|^{k-1}}|X_1||X_2|\cdots |X_k|\,\right\}.
\]
\end{theorem}

\subsection{Applications} \label{sec_rm}
Let $q$ be an odd prime power. We now take $N_0=\bF_q^2$ and $H_0=SO_2(\bF_q)$. Define the group homomorphism $\varphi:\, H_0\rightarrow \textsf{Aut}(N_0)$ by $\varphi(h):=\varphi_h$ with $\varphi_h(x) = hx$. Take the semi-direct product $(G_0,\cdot)= (N_0,+) \rtimes_\varphi (H_0,\cdot)$. The set on the right-hand side is
\[
\{(z,h):\, z\in N_0,\, h\in H_0\},
\]
with the group law given by 
\[
(z_1,h_1)(z_2,h_2) = (z_1+h_1z_2,\, h_1h_2).
\]


The group $G_0$ consists of the rigid-motions that preserve `distances' between pairs of points. For any $(z,h)\in G_0$, this affine transformation is given by $(z,h):\, N_0\rightarrow N_0$ with
\[
(z,h)x =z+hx.
\]

The group $H_0$ is given explicitly by
\[
H_0=\left\{\left[\begin{matrix}
a &-b\\
b &a
\end{matrix}\right]:\, a^2+b^2=1\right\}.
\]
It is a cyclic group of order $q-\varepsilon_q$, where 
\[
\varepsilon_q = \left(\frac{-1}{q}\right)=
\begin{cases}
1,\quad &\text{if }q\equiv 1\, (\mod 4),\\
-1,\quad &\text{if }q\equiv 3\, (\mod 4).
\end{cases}
\]
In the next result, we show that $G_0$ is $(q-\varepsilon_q)$-quasirandom on $N_0$.
\begin{theorem}\label{thm_repn_dim_H0}
The group $G_0$ has $q-\varepsilon_q$ type I irreducible complex representations of degree $1$, and has $q+\varepsilon_q$ type II irreducible complex representations of degree $q-\varepsilon_q$.
\end{theorem}
\begin{remark}\label{rm2}
One can similarly consider the group $G_1=\bF_q^2 \rtimes_\varphi SL_2(\bF_q)$, which is composed of affine maps that are area--triangle--invariant. It is not hard to prove that $G_1$ is $(q-1)/2$-quasirandom by using Theorem \ref{thm_repn_dim_H0}. Indeed, all complex representations of $G_1$ of type II has degree no less than $q-\varepsilon_q$, and all non-trivial representations of $G_1$ of type I has dimension no less than $(q-1)/2$.
\end{remark}


Note that $|G_0| = (q-\varepsilon_q)q^2$. Applying Theorems \ref{thm_mixing} and \ref{thm_lowbound_XY} on the group $G_0$, and combining Theorem \ref{thm_repn_dim_H0}, we obtain the following two theorems immediately. 

\begin{theorem} \label{thm_mixing_G0}
For any $X,Y,Z\subseteq G_0$, we have 
\begin{equation}\label{eq1new}
\left|\#\{(x,y,z)\in X\times Y\times Z:\, xy=z\} -  \frac{\dot{\cM_2}}{q^2}\right|\leq q\sqrt{|X||Y||Z|},
\end{equation}
where 
\[
\dot{\cM_2}=   \#\{(x,y,z)\in X\times Y\times Z:\, \dot{x}\dot{y}=\dot{z}\}.
\]
\end{theorem}

\begin{theorem} \label{thm_lowXY_G0}
For $X,Y\subseteq G_0$, we have  
\begin{equation}\label{eq2new}
|XY| \gg \min \left\{\frac{q^2|X|^2|Y|^2}{\#\{(g_1,g_2,h_1,h_2)\in X^2\times Y^2:\, \dot{g_1}\dot{h_1}=\dot{g_2}\dot{h_2}\}},\, \frac{|X||Y|}{q^2}\right\}. 
\end{equation}
\end{theorem}

When the set $X$ is close to a large subgroup of $G$, the error term $q|X|^{3/2}$ in \eqref{eq1new} and the original main term $q^{-1}|X|^3$ in \eqref{one} may be both smaller than the number of triples $(x,y,z)\in X^3$ taken into consideration. So the refinement of the main term in \eqref{eq1new} is necessary. Moreover, the second term $q^{-2}|X|^2$ on the right-hand side of \eqref{eq2new} may be larger than $|XX|$. So the replacement of the first term on the right-hand side of \eqref{eq_prodgrowth} by that of \eqref{eq2new} is also necessary. The details are given in the following example, which also shows the sharpness of Theorem \ref{thm_lowXY_G0}.

\begin{example} \label{ex1}
Let $\gamma$ be a generator of the cyclic group $H_0$. Assume that $|H_0|=q-\varepsilon_q=kl$, $a=\gamma^k$, and let 
\[
\cX=\{(t,a^j):\, t\in N_0,\, 0\leq j\leq l-1\}.
\]
It is not hard to see that $\cX$ is a subgroup of $G$. So, 
\[
|\cX\cX|=|\cX|= lq^2 .
\]
and
\[
\#\{(x,y,z)\in \cX^3:\, xy=z\} = |\cX|^2 = l^2q^4.
\]
The error term in \eqref{eq1new} is $q|\cX|^{3/2}=l^{3/2}q^4 \leq l^2q^4$, and the original main term in \eqref{one} is about $q^{-3}|\cX|^3 = l^3 q^3\leq l^2q^4$. And the second term on the right-hand side of \eqref{eq2new} is $q^{-2}|\cX|^2 = l^2q^2\geq lq^2$.

To calculate the quantities 
\[
\dot{\cM_2}= \#\{(x,y,z)\in \cX^3:\, \dot{x}\dot{y}=\dot{z}\}~~\mbox{and}~~ \dot{\cN_2}=\#\{(g_1,g_2,h_1,h_2)\in \cX^4:\, \dot{g_1}\dot{h_1}=\dot{g_2}\dot{h_2}\},
\]
we see that for given $x, y\in \cX$, one has $\dot{x}\dot{y}\in \{a^j:\, 0\leq j\leq l-1\}$ and $
\#\{z\in X:\, \dot{z}=\dot{x}\dot{y}\} = |N_0|.
$
It follows that $\dot{\cM_2}=|\cX|^2|N_0|$. Similarly, $\dot{\cN_2}=|\cX|^3|N_0|$.
Thus, Theorems \ref{thm_mixing_G0} and \ref{thm_lowXY_G0} show that 
\[
|l^2q^4 - l^2q^4| \leq l^{3/2}q^4~~\mbox{ and }~~ lq^2 \,\gg \, \min\{ l q^2,\, l^2q^2\}.
\]
In particular, Theorem \ref{thm_lowXY_G0} is optimal in this example.
\end{example} 
\begin{remark}
If the reader is interested in analogs of Theorems \ref{thm_mixing_G0} and \ref{thm_lowXY_G0} in the setting of the group $G_1=\bF_q^2\rtimes_\varphi SL_2(\bF_q)$, then similar results can be derived in the same way by using Remark \ref{rm2}. 
\end{remark}

Denote by $\rho_0,\rho_1,\ldots,\rho_{q-\varepsilon_q-1}$ the irreducible representations of $G_0$ of type I. When $X=Y$, the next theorem offers a lower bound in terms of the Fourier bias of the set $X$. 

\begin{theorem} \label{thm_bias}
For $X\subseteq G_0$, assume that 
\[
|\widehat{1_X}(\rho_r)| \leq M
\]
for all but $k$ indices $r\in \{0,1,2,\ldots, q-\varepsilon_q-1\}$. Then 
\[
|XX| \gg  \min \left\{\frac{q^3}{k}, \,\frac{|X|^4}{q^{10}M^4},\, \frac{|X|^2}{q^2}\right\}. 
\]
When $M=0$, the second term on the right-hand side of above formula can be omitted.
\end{theorem}
Example \ref{ex1} also shows that this theorem is sharp. The details are given in Example \ref{ex2}.

In the next theorem, we prove that if the set $X$ satisfies certain properties, then the product $XX$ grows exponentially.

For two points $x=(x_1,x_2)$ and $y=(y_1,y_2)$ in $\bF_q^2$, define
\[
\|x-y\| = (x_1-y_1)^2+(x_2-y_2)^2.
\]
We say that $(x,y)\in \bF_q^2$ is a line segment of length $t$ if $\|x-y\|=t$. Note that rigid motions map line segments to line segments of same length. Now let $E$ be a given subset of $\mathbb{F}_q^2$ with $|E|>q^{3/2}$. For $t\ne 0$, let $n_t$ be the number of line segments $(x, y)\in E\times E$ such that $||x-y||=t$. Iosevich and Rudnev \cite{IR} proved that $n_t=(1+o(1))|E|^2/q$, here $o(1)\to 0$ as $q\to \infty$. Fix a line segment $(u_0, v_0)\in E\times E$ of length $t$. For each $(x, y)\in E\times E$ such that $||x-y||=t$, there exists a unique rigid motion $(z,\theta)\in G_0$ such that $(z,\theta)x=u_0$ and $(z,\theta)y=v_0$. Let $X_t$ be the set of all corresponding rigid motions when $(x, y)\in E\times E$ runs over all pairs of length $t$. Then $|X_t|=n_t =(1+o(1))|E|^2/q$.

\begin{theorem}\label{group_gen_distance}
Let $E\subseteq \mathbb{F}_q^2$ with $|E|=q^\alpha$ and  $\alpha\in (3/2, 2)$. Let $t\in \bF_q^\ast$. Then there exists $\delta=\delta(\alpha)>0$ such that \[|X_tX_t|\gg |X_t|^{1+\delta}.\]
In particular, we have 
\[|X_tX_t| \gg \min\left\lbrace q|E|,~\frac{|E|^4}{q^4}\right\rbrace.\]
\end{theorem}
In the above theorem, $|X_tX_t|\gg q^3$ if and only if $|E|\gg q^2$. This raises the following question: does there exist $\epsilon>0$ such that $|X_tX_t|\gg q^3$ whenever $|E|\gg q^{2-\epsilon}$?
\section{Preliminaries}

In this section, we recall some basic properties of semi-driect products and non-abelian Fourier analysis.

In $G=N\rtimes_\varphi H$, the identity element is $(1_N,1_H)$, and the inverse is given by 
\[
(z,h)^{-1}= (\varphi_{h^{-1}}(z^{-1}), h^{-1}).
\]
The product of two elements $(z_1, h_1)$ and $(z_2, h_2)$ is determined by 
\[
(z_1,h_1)(z_2,h_2) = (z_1\,\varphi_{h_1}(z_2),\, h_1h_2).
\]
Recall the notation $g:=(\ddot{g},\dot{g})=(g^{\cdot\cdot}, g^{\cdot})$ for $g\in G$. It follows from 
\[
g_1g_2 = (\ddot{g_1},\dot{g_1})(\ddot{g_2},\dot{g_2}) = (\ddot{g_1}\,\varphi_{\dot{g_1}}(\ddot{g_2}),\, \dot{g_1}\dot{g_2}), \quad g^{-1} = (\ddot{g},\dot{g})^{-1} = (\varphi_{\dot{g}^{-1}}(\ddot{g}^{-1}),\dot{g}^{-1})
\]
that $(g_1g_2)^\cdot = \dot{g_1}\dot{g_2}$ and $(g^{-1})^\cdot = (\dot{g})^{-1}$.

Let $\widehat{G}$ and $\widehat{H}$ be the sets of all (non-isomorphic) irreducible complex representations of $G$ and $H$, respectively. 
Let $Q$ and $Q'$ be positive integers. Assume that $\widehat{H}=\{\dot{\rho_0},\dot{\rho_1},\ldots,\dot{\rho}_{Q-1}\}$, here $\dot{\rho_0}$ is the trivial representation. Then all irreducible representations of $G$ of type I is given by $\widehat{G}_I = \{\rho_0,\rho_1,\ldots,\rho_{Q-1}\}$, where 
\[
\rho_r(g) = \dot{\rho_r}(\dot{g}),\quad (g\in G,\,\, 0\leq r\leq Q-1).
\]
Here $\rho_0$ is the trivial representation. We also denote the set of the irreducible representations of $G$ of type II by $\widehat{G}_{II}=\{\rho'_1,\rho_2',\ldots,\rho_{Q'}'\}$. Therefore, 
\[\widehat{G}=\{\rho_0, \ldots, \rho_{Q-1},\rho_1', \ldots, \rho_{Q'}' \}.\]

For each representation $\rho$, its degree will be denoted by $d_\rho$, and its character, denoted by $\chi_{\rho}\colon G\to \mathbb{C}$, is defined by $\chi_{\rho}(x)=\textsf{Tr}(\rho(x))$.

The main tool is non-abelian Fourier analysis. For each representation $(\rho,V)$ with $\rho\in \widehat{G}$, we may assume that the inner product on $V$ is chosen so that $\rho$ is unitary. Assume that the inner product and norm on $V$ is given by 
\[
<a,b>_{\HS}=\textsf{Tr}(b^\ast a),\quad \|a\|_{\HS}^2 = \<a,a\>_{\HS} = \textsf{Tr}(a^\ast a). 
\]
It satisfies the properties that
\[
|\<a,b\>|_{\HS}\leq \|a\|_{\HS}\|b\|_{\HS},
\]
and 
\[
\|ab\|_{\HS}\leq \|a\|_{\HS}\|b\|_{\HS}.
\]

For a function $f:\, G\rightarrow \bC$, we use the notation.
\[
\bE_{x\in G} f(x) = \frac{1}{|G|}\sum\limits_{x\in G}f(x). 
\]
The Fourier transformation is defined by
\[
\widehat{f}(\rho) = \bE_{x\in G} f(x)\rho(x),\quad (\rho\in \widehat{G}),
\]
and the Fourier inverse is given by
\[
f(x) = \sum\limits_{\rho\in \widehat{G}} d_\rho \<\widehat{f}(\rho),\rho(x)\>_{\HS},\quad (x\in G).
\]
The Parseval's identity is 
\[
\<f_1,f_2\> = \sum\limits_{\rho\in \widehat{G}} d_\rho \<\widehat{f_1}(\rho),\widehat{f_2}(\rho)\>_{\HS},
\]
where the inner product of two functions $f$ and $g$ on $G$, denoted by $\<f, g\>$, is 
$\<f,g\> = \bE_{x\in G}f(x)\overline{g(x)}.$

The convolution of two functions $f_1$ and $f_2$ on $G$ is defined by 
\[
(f_1\ast f_2)(x) = \bE_{y\in G} f_1(xy^{-1})f_2(y).
\]
We have the property that $\widehat{f_1\ast f_2}=\widehat{f_1}\widehat{f_2}$. For $p\geq 1$, the $l^p$-norm of $f$ is given by 
\[
\|f\|_p = \left(\bE_{x\in G} |f(x)|^p\right)^{1/p}.
\] 
If $1/r=1/p+1/q$, then we have the H\"{o}lder inequality $
\|fg\|_r \leq \|f\|_p\|g\|_q$.
For a subset $X\subseteq G$, one has $\|1_X\|_1=\|1_X\|_2^2=|X|/|G|$.

\section{Proof of Theorems \ref{thm_mixing}, \ref{thm_l_2}, and \ref{thm_lowbound_XY}}
\begin{proof} [Proof of Theorem \ref{thm_mixing}]
Note that the characteristic functions $1_{X_j}$ $(0\leq j\leq k)$ all take non-negative value. The number of solutions is counted by 
\begin{eqnarray*}
&&\frac{1}{|G|^k}\#\{(x_0,x_1,\ldots,x_k)\in X_0\times X_1\times \ldots\times X_k:\, x_1x_2\ldots x_k=x_0\} \\
&&= \bE_{x\in G}(1_{X_1}\ast 1_{X_2}\ast \ldots \ast 1_{X_k})(x)1_{X_0}(x) \\
&&= \<1_{X_1}\ast 1_{X_2}\ast \ldots \ast 1_{X_k},1_{X_0}\>\\
&& = \sum\limits_{\rho\in \widehat{G}}d_\rho\big\<\reallywidehat{(1_{X_1}\ast 1_{X_2}\ast \ldots \ast 1_{X_k})}(\rho),\, \widehat{1_{X_0}}(\rho)\big\>_{\HS},
\end{eqnarray*}
where the Parseval's identity is applied. Now we split the sum into two parts:
\[
\fM = \sum\limits_{r=0}^{Q-1} d_{\rho_r}\<\reallywidehat{(1_{X_1}\ast 1_{X_2}\ast \ldots \ast 1_{X_k})}(\rho_r),\, \widehat{1_{X_0}}(\rho_r)\>_{\HS},\quad \fE=\sum\limits_{j=1}^{Q'} d_{\rho'_j}\<\reallywidehat{(1_{X_1}\ast 1_{X_2}\ast \ldots \ast 1_{X_k})}(\rho_j'),\, \widehat{1_{X_0}}(\rho_j')\>_{\HS}.
\]
For type II representions, we have
\begin{eqnarray*}
|\fE|&& \leq  \sum\limits_{j=1}^{Q'} d_{\rho'_j}\|\reallywidehat{(1_{X_1}\ast 1_{X_2}\ast \ldots \ast 1_{X_k})}(\rho_j')\|_{\HS}\, \|\widehat{1_{X_0}}(\rho_j')\|_{\HS}\\
&&\leq  \left(\sum\limits_{j=1}^{Q'} d_{\rho'_j}\|\reallywidehat{(1_{X_1}\ast 1_{X_2}\ast \ldots \ast 1_{X_k})}(\rho_j')\|_{\HS}^2\right)^{1/2}\left(\sum\limits_{j=1}^{Q'} d_{\rho'_j}\|\widehat{1_{X_0}}(\rho_j')\|_{\HS}^2\right)^{1/2}.
\end{eqnarray*}
Since $d_{\rho_j'}\geq D$ $(1\leq j\leq Q')$, one has
\begin{eqnarray*}
&&\sum\limits_{j=1}^{Q'} d_{\rho'_j}\|\reallywidehat{(1_{X_1}\ast 1_{X_2}\ast \ldots \ast 1_{X_k})}(\rho_j')\|_{\HS}^2\\
&& = \sum\limits_{j=1}^{Q'} d_{\rho'_j}\|\widehat{1_{X_1}}(\rho_j')\widehat{1_{X_2}}(\rho_j')\ldots\widehat{1_{X_k}}(\rho_j')\|_{\HS}^2 \\
&&\leq \frac{1}{D^{k-1}}\sum\limits_{j=1}^{Q'} d_{\rho'_j}^k\|\widehat{1_{X_1}}(\rho_j')\|_{\HS}^2 \,\|\widehat{1_{X_2}}(\rho_j')\|_{\HS}^2\ldots \|\widehat{1_{X_k}}(\rho_j')\|_{\HS}^2 \\
&&\leq  \frac{1}{D^{k-1}}\left(\sum\limits_{j=1}^{Q'} d_{\rho'_j}\|\widehat{1_{X_1}}(\rho_j')\|_{\HS}^2\right)\left(\sum\limits_{j=1}^{Q'} d_{\rho'_j}\|\widehat{1_{X_2}}(\rho_j')\|_{\HS}^2\right)\ldots \left(\sum\limits_{j=1}^{Q'} d_{\rho'_j}\|\widehat{1_{X_k}}(\rho_j')\|_{\HS}^2\right)\\
&&\leq \frac{1}{D^{k-1}}\left(\sum\limits_{\rho\in \widehat{G}} d_{\rho}\|\widehat{1_{X_1}}(\rho)\|_{\HS}^2\right)\left(\sum\limits_{\rho\in \widehat{G}} d_{\rho}\|\widehat{1_{X_2}}(\rho)\|_{\HS}^2\right)\ldots \left(\sum\limits_{\rho\in \widehat{G}} d_{\rho}\|\widehat{1_{X_k}}(\rho)\|_{\HS}^2\right)\\
&& = \frac{1}{D^{k-1}} \|1_{X_1}\|_2^2\|1_{X_2}\|_2^2\ldots \|1_{X_k}\|_2^2 = \frac{|X_1||X_2|\ldots |X_k|}{D^{k-1}|G|^k}. 
\end{eqnarray*}
Similarly, 
\[
\sum\limits_{j=1}^{Q'} d_{\rho'_j}\|\widehat{1_{X_0}}(\rho_j')\|_{\HS}^2 \leq \|1_{X_0}\|_2^2 = \frac{|{X_0}|}{|G|}.
\]
Thus,
\[
\fE \leq \sqrt{\frac{|{X_0}||{X_1}|\ldots |{X_k}|}{D^{k-1}|G|^{k+1}}}.
\]

For type I represetations, we have 
\begin{eqnarray*}
\fM&&=\sum\limits_{r=0}^{Q-1} d_{\rho_r}\<\reallywidehat{(1_{X_1}\ast 1_{X_2}\ast \ldots \ast 1_{X_k})}(\rho_r),\, \widehat{1_{X_0}}(\rho_r)\>_{\HS} \\
&&= \sum\limits_{r=0}^{Q-1}d_{\rho_r}\bE_{u\in G}\bE_{v\in G}(1_{X_1}\ast 1_{X_2}\ast \ldots \ast 1_{X_k})(u)1_{X_0}(v)\<\rho_r(u),\rho_r(v)\>_{\HS}\\
&&=\bE_{u\in G}\bE_{v\in G}(1_{X_1}\ast 1_{X_2}\ast \ldots \ast 1_{X_k})(u)1_{X_0}(v) \sum\limits_{r=0}^{Q-1}d_{\rho_r}\chi_{\rho_r}(v^{-1}u).
\end{eqnarray*}
For given $u,v\in G$, we have $(v^{-1}u)^\cdot = \dot{v}^{-1}\dot{u}$. Then
\begin{eqnarray*}
\sum\limits_{r=0}^{Q-1}d_{\rho_r}\chi_{\rho_r}(v^{-1}u) = \sum\limits_{r=0}^{Q-1}d_{\dot{\rho_r}}\chi_{\dot{\rho_r}}(\dot{v}^{-1}\dot{u}) = \sum\limits_{\rho\in \widehat{H}}d_\rho\chi_{\rho}(\dot{v}^{-1}\dot{u}) = 
\begin{cases}
|H|,\quad &\text{if }\dot{u}= \dot{v},\\
0,&\text{if }\dot{u}\neq \dot{v}.
\end{cases}
\end{eqnarray*}
Now we have
\begin{eqnarray*}
\fM&&= \frac{|H|}{|G|^2}\sum\limits_{u,v\in G\atop \dot{u}=\dot{v}}(1_{X_1}\ast 1_{X_2}\ast \ldots \ast 1_{X_k})(u)1_{X_0}(v) \\
&&= \frac{|H|}{|G|^{k+1}} \#\{(x_0,x_1,\ldots,x_k)\in X_0\times X_1\times\ldots\times X_k:\, \dot{x_1}\dot{x_2}\ldots \dot{x_k}=\dot{x_0}\}. 
\end{eqnarray*}
Combining all above formulae, the Theorem \ref{thm_mixing} then follows. 
\end{proof}

\begin{proof}[Proof of Theorem \ref{thm_l_2}]
Note that 
\begin{eqnarray*}
&& \|1_{X_1}\ast 1_{X_2}\ast \ldots \ast 1_{X_k}\|_2^2= \sum\limits_{\rho\in \widehat{G}} d_\rho\|\reallywidehat{(1_{X_1}\ast 1_{X_2}\ast \ldots \ast 1_{X_k})}(\rho_r)\|_{\HS}^2= \sum\limits_{\rho\in \widehat{G}} d_\rho \|\widehat{1_{X_1}}(\rho)\widehat{1_{X_2}}(\rho)\ldots \widehat{1_{X_k}}(\rho)\|_{\HS}^2.
\end{eqnarray*}
We use similar argument as in the proof of Theorem \ref{thm_mixing}. Let
\[
\fM = \sum\limits_{r=0}^{Q-1} d_{\rho_r}\|\reallywidehat{(1_{X_1}\ast 1_{X_2}\ast \ldots \ast 1_{X_k})}(\rho_r)\|_{\HS}^2,\quad \fE=\sum\limits_{j=1}^{Q'} d_{\rho'_j}\|\widehat{1_{X_1}}(\rho'_j)\widehat{1_{X_2}}(\rho'_j)\ldots \widehat{1_{X_k}}(\rho'_j)\|_{\HS}^2.
\]
For type II representions, we have
\begin{eqnarray*}
0\leq \fE &&\leq \frac{1}{D^{k-1}}\sum\limits_{j=1}^{Q'} \left(d_{\rho'_j}\|\widehat{1_{X_1}}(\rho'_j)\|_{\HS}^2\right)\left(d_{\rho'_j}\|\widehat{1_{X_2}}(\rho'_j)\|_{\HS}^2\right)\ldots \left(d_{\rho'_j}\|\widehat{1_{X_k}}(\rho'_j)\|_{\HS}^2\right)\\
&&\leq \frac{1}{D^{k-1}}\left(\sum\limits_{j=1}^{Q'}d_{\rho'_j}\|\widehat{1_{X_1}}(\rho'_j)\|_{\HS}^2\right)\left(\sum\limits_{j=1}^{Q'}d_{\rho'_j}\|\widehat{1_{X_2}}(\rho'_j)\|_{\HS}^2\right)\ldots \left(\sum\limits_{j=1}^{Q'}d_{\rho'_j}\|\widehat{1_{X_k}}(\rho'_j)\|_{\HS}^2\right)\\
&&\leq \frac{1}{D^{k-1}}\|1_{X_1}\|_2^2\|1_{X_2}\|_2^2\ldots \|1_{X_k}\|_2^2 = \frac{|X_1||X_2|\ldots |X_k|}{D^{k-1}|G|^k}.
\end{eqnarray*}

For type I represetations, we have 
\begin{eqnarray*}
\fM&&=\sum\limits_{r=0}^{Q-1} d_{\rho_r}\<\reallywidehat{(1_{X_1}\ast 1_{X_2}\ast \ldots \ast 1_{X_k})}(\rho_r),\, \reallywidehat{(1_{X_1}\ast 1_{X_2}\ast \ldots \ast 1_{X_k})}(\rho_r)\>_{\HS} \\
&&= \sum\limits_{r=0}^{Q-1}d_{\rho_r}\bE_{u\in G}\bE_{v\in G}(1_{X_1}\ast 1_{X_2}\ast \ldots \ast 1_{X_k})(u)(1_{X_1}\ast 1_{X_2}\ast \ldots \ast 1_{X_k})(v)\<\rho_r(u),\rho_r(v)\>_{\HS}\\
&&=\bE_{u\in G}\bE_{v\in G}(1_{X_1}\ast 1_{X_2}\ast \ldots \ast 1_{X_k})(u)(1_{X_1}\ast 1_{X_2}\ast \ldots \ast 1_{X_k})(v) \sum\limits_{r=0}^{Q-1}d_{\rho_r}\chi_r(v^{-1}u)\\
&&= \frac{|H|}{|G|^2}\sum\limits_{u,v\in G\atop \dot{u}=\dot{v}}(1_{X_1}\ast 1_{X_2}\ast \ldots \ast 1_{X_k})(u)(1_{X_1}\ast 1_{X_2}\ast \ldots \ast 1_{X_k})(v) \\
&&= \frac{|H|}{|G|^{2k}} \#\{(x_1,x_2,\ldots,x_k,y_1,y_2,\ldots,y_k)\in (X_1\times X_2\times\ldots\times X_k)^2:\, \dot{x_1}\dot{x_2}\ldots\dot{x_k}=\dot{y_1}\dot{y_2}\ldots\dot{y_k}\}. 
\end{eqnarray*}
The theorem now follows.
\end{proof}

The proof of Theorem \ref{thm_lowbound_XY} is based on the following lemma.
\begin{lemma} \label{lem}
Let $G$ be a group and $Z\subseteq G$. Suppoes that $f:G\rightarrow \bC$ is a function whose support is contained in $Z$. Then  
\[
|Z| \geq |G|\frac{\|f\|_1^2}{\|f\|_2^2}
\]
\end{lemma}

\begin{proof}
By Cauchy-Schwartz inequality, we have 
\begin{eqnarray*}
\|f\|_1^2 = \left(\bE_{x\in G}|f(x)|\right)^2 = \left(\frac{1}{|G|}\sum\limits_{x\in Z}|f(x)|\right)^2 \leq \left(\frac{1}{|G|}\sum\limits_{x\in Z}1^2\right)\left(\frac{1}{|G|}\sum\limits_{x\in Z}|f(x)|^2\right) = \frac{|Z|}{|G|}\|f\|_2^2.
\end{eqnarray*}
The proof is completed.
\end{proof}

\begin{proof}[Proof of Theorem \ref{thm_lowbound_XY}]
Note that the support of $1_{X_1}\ast 1_{X_2}\ast \ldots \ast 1_{X_k}$ is exactly $X_1X_2\ldots X_k$. By Lemma \ref{lem}, we have 
\[
|X_1X_2\ldots X_k|  \geq |G|\frac{\|1_{X_1}\ast 1_{X_2}\ast \ldots \ast 1_{X_k}\|_1^2}{\|1_{X_1}\ast 1_{X_2}\ast \ldots \ast 1_{X_k}\|_2^2}.
\] 
The $l^1$-norm can be computed directly by  
\begin{equation} \label{eq_l1}
\|1_{X_1}\ast 1_{X_2}\ast \ldots \ast 1_{X_k}\|_1 = \frac{1}{|G|^k}\sum\limits_{x\in G} \sum\limits_{x_1,\ldots,x_k\in G\atop x_1x_2\ldots x_k = x}1_{X_1}(x_1)1_{X_2}(x_2)\ldots 1_{X_k}(x_k)
=  \frac{|X_1||X_2|\ldots |X_k|}{|G|^k}.
\end{equation}
Combing Theorem \ref{thm_l_2}, Theorem \ref{thm_lowbound_XY} then follows. 
\end{proof}

\section{Proof of Theorem \ref{thm_repn_dim_H0}}

In the rest of this paper, we always write $Q=q-\varepsilon_q=|H_0|$ and $Q'= q+\varepsilon_q$. Then $QQ'=q^2-1$.

For $G_0=N_0\rtimes_\varphi H_0$ with $N_0=\bF_q^2$ and $H_0= SO_2(\bF_q)$, the group operation is given by 
\[
(z_1,h_1)(z_2,h_2) = (z_1+h_1z_2,\, h_1h_2).
\]
We also have the identity element $(0,1)$, where $0$ is the zero vector in $N_0$ and $1$ is the identity matrix in $H$, and the inverse of $(z, h)$ is
$
(z,h)^{-1}= (-h^{-1}z, h^{-1}).
$
Recall that $H_0$ is abelien, a conjugation of elements is given by
\[
(w,k)(z,h)(w,k)^{-1} = (w+kz,kh)(-k^{-1}w,k^{-1}) = (kz+(1-h)w,h). 
\]

\begin{lemma} \label{lem_invertible}
For any $h\in H_0$ with $h\neq 1$, the matrix $1-h$ is ivertible.
\end{lemma}

\begin{proof}
Let 
\[
h=\left[\begin{matrix}
a &-b\\
b &a
\end{matrix}\right],
\]
where $a^2+b^2=1$. We have
\[
\textsf{det}(1 - h) = \left|\begin{matrix}
1-a &b\\
-b &1-a
\end{matrix}\right| = 1-2a+a^2+b^2 = 2(1-a).
\]
Since $h$ is not the identity matrix, we have $a\neq 1$ and $\text{det}(1-h)\neq 0$. So $1-h$ is invertible. 
\end{proof}

Recall that $\varphi:\, H_0\rightarrow Aut(N_0)$ is given by $\varphi_h(x)=hx$. By Lemma \ref{lem_invertible}, we have 
\[
\textsf{Stab}_{H_0}(x) = \{h\in H_0:\, hx=x\} = \{1\}
\]
for $x\in N_0\setminus \{0\}$. It follows that the cardinality of the orbit of $x$ is
\begin{equation} \label{eq_orbit}
|\textsf{Orb}_{H_0}(x)| = \frac{|H_0|}{|\textsf{Stab}_{H_0}(x)|} = Q. 
\end{equation}
Indeed, apart from the orbit of $0$ with length one, there are $Q'$ disjoint orbits with length $Q$. It satiesfies that $|N_0| = q^2 = 1+QQ'$.

\begin{proof} [Proof of Theorem \ref{thm_repn_dim_H0}]
The cardinality of $G_0$ is $|G_0|=Qq^2$. Since $H_0$ is an cyclic group, the type I irreducible representations of $G_0$ are lifted from representations of $H_0$, and all have degree $1$. More explicitly, let $\gamma$ be a generator of $H_0$, i.e., $H_0=\{\gamma^j:\, 0\leq j\leq Q-1\}$, then the representations can be given by 
\begin{equation} \label{eq_rep_detail}
\rho_r\big((z,\gamma^j)\big) = \dot{\rho_r}(\gamma^j)= e^{2\pi \i \frac{rj}{Q}},\quad (z\in N,\, 0\leq j\leq Q-1).
\end{equation}
Moreover, since $\widetilde{N_0}=\{(z,1):\, z\in N\}$ is an abelian subgroup of $G_0$ with $[G_0:\, \widetilde{N_0}]=|H_0|$, the dimension of any irreducible representation of $G_0$ is no larger than $|H_0|$.  

In the following, we will find all conjugacy classes of $G_0$ and then determine the degree of type II representations. For $(z,1)\in G_0$ with $z\neq 0$, 
\[
(w,k)(z,1)(w,k)^{-1} = (kz,1)
\]
for any $(w,k)\in G_0$. Here $kz\neq 0$, since $k$ is invertible. Now the conjugacy class of $(z,1)$ is
\[
[(z,1)] = \{(kz,1):\, k\in H_0\}.
\]
It follows from \eqref{eq_orbit} that 
\[
\#[(z,1)]=|\textsf{Orb}_{H_0}(z)| =Q.
\]
Indeed, the set $\widetilde{N_0}\setminus \{(0,1)\}$ is divided into $Q'$ disjoint conjugacy classes, each of cardinality $Q$. Moreover, $[(0,1)]=\{(0,1)\}$ is another conjugacy class. 

Next we consider the conjugacy class of $(0,h)$ for some $h\in H_0\setminus \{1\}$. One has
\[
(w,k)(0,h)(w,k)^{-1}  = \big((1-h)w,h\big) 
\]
for any $(w,k)\in G_0$. Since $1-h$ is invertible by Lemma \ref{lem_invertible}, one has
\[
[(0,h)] = \{(y,h):\, y\in N_0\}.
\]
There are $Q-1$ such conjugacy classes. 

Now, let $\rho_0$, $\rho_r$ $(r=1,2,\ldots, q-\varepsilon_q-1)$ and $\rho_j^\prime$ $(j=1,2,\ldots, Q')$ be the irreducible representations of $G_0$ corresponding to the conjugacy classes $[(0,1)]$, $[(0,h)]$ $(h\neq 1)$ and $[(z,1)]$ $(z\neq 0)$. Then 
\[
Q^2Q'\geq \sum\limits_{j=1}^{Q'} d_{\rho_i^\prime}^2 = |G|^2 - \sum\limits_{r=0}^{Q-1}d_{\rho_r}^2 =  Qq^2 - Q = Q^2Q'. 
\]
It follows that the equality holds, i.e., $d_{\rho_i^\prime}=Q$ for all $j=1,2,\ldots,Q'$. The proof is completed. 
\end{proof}


\section{Proof of Theorems \ref{thm_bias} and \ref{group_gen_distance}}
\begin{proof}[Proof of Theorem \ref{thm_bias}] Recall that $Q=q-\epsilon_q$. For any $0\leq r\leq Q-1$, we have $|\widehat{1_X}(\rho_r)|\leq \|1_X\|_1 = |X|/|G_0|$. Assume that $\{0,1,\ldots,Q-1\}=\Gamma_0\sqcup\Gamma_1$, where $|\Gamma_0|=k$ and $\sup_{r\in \Gamma_1}|\widehat{1_X}(\rho_r)| \leq M$. We now split the sums
\[
\|1_X\ast 1_X\|_2^2 = \sum\limits_{\rho\in \widehat{G_0}} d_\rho \|\widehat{1_X\ast 1_X}(\rho)\|_{\HS}^2 = \sum\limits_{\rho\in \widehat{G_0}}d_\rho \|\widehat{1_X}(\rho)^2\|_{\HS}^2
\]
into three parts
\[
\Sigma_0 = \sum\limits_{r\in \Gamma_0} d_{\rho_r}  |\widehat{1_X}(\rho_r)|^4 \leq k\frac{|X|^4}{|G_0|^4},~~
\Sigma_1  \leq \sum\limits_{r\in \Gamma_1} d_{\rho_r} |\widehat{1_X}(\rho_r)|^4 \leq  (Q-k)M^4,
\]
and 
\begin{eqnarray*}
\Sigma_2 &&= \sum\limits_{j=1}^{Q'} d_{\rho_j^\prime}\|\widehat{1_X}(\rho_j^\prime)^2\|_{\HS}^2 \leq \frac{1}{D}\sum\limits_{j=1}^{Q'} d^2_{\rho_j^\prime}\|\widehat{1_X}(\rho_j^\prime)\|_{\HS}^2\|\widehat{1_X}(\rho_j^\prime)\|_{\HS}^2\\
&&\leq \frac{1}{D}\left(\sum\limits_{j=1}^{Q'} d_{\rho_j^\prime}\|\widehat{1_X}(\rho_j^\prime)\|_{\HS}^2\right)^2 \leq \frac{1}{D}\left(\|1_X\|_2^2\right)^2 = \frac{1}{D}\frac{|X|^2}{|G_0|^2}.
\end{eqnarray*}
Thus, 
\[
\|1_X\ast 1_X\|_2^2 \leq k\frac{|X|^4}{|G_0|^4}+ (Q-k)M^4+\frac{1}{D}\frac{|X|^2}{|G_0|^2},
\]
Using Theorem \ref{thm_repn_dim_H0}, Lemma \ref{lem} and \eqref{eq_l1}, one obtains
\[
|XX| \gg \min \left\{\frac{|G_0|}{k}, \,\frac{|X|^4}{|G_0|^3(Q-k)M^4},\, \frac{D|X|^2}{|G_0|}\right\} \gg \min \left\{\frac{q^3}{k}, \,\frac{|X|^4}{q^{10}M^4},\, \frac{|X|^2}{q^2}\right\}. 
\]
The theorem then follows.
\end{proof}

\begin{example} \label{ex2}
Let $\cX$ be the set chosen in Example \ref{ex1}. Recall that $Q=q-\varepsilon_q=kl$. For $0\leq r\leq Q-1$, by \eqref{eq_rep_detail} we have 
\begin{eqnarray*}
\widehat{1_{\cX}}(\rho_r) = \bE_{x\in G}1_{\cX}(x)\rho_r(x) = \frac{1}{|G_0|}\sum\limits_{z\in N}\sum\limits_{j=0}^{l-1} e^{2\pi \i \frac{rk}{Q} j} \\
= \frac{1}{|G_0|}\sum\limits_{z\in N}\sum\limits_{j=0}^{l-1} e^{2\pi \i \frac{r}{l} j} = \begin{cases}\frac{1}{k},\quad &\text{if }l|r,\\0,&\text{if }l\nmid r.\end{cases}
\end{eqnarray*}
As a result, in Theorem \ref{thm_bias}, we can take $M=0$ for all but $Q/l=k$ irreducible representations of dimension $1$, and obtain
\[
\frac{q^3}{k}=|\cX|=|\cX\cX|\gg \min\left\{\frac{q^3}{k},\, \frac{q^4}{k^2}\right\} \gg \frac{q^3}{k},
\]
which gives the correct order.
\end{example}

\begin{proof}[Proof of Theorem \ref{group_gen_distance}]
We first start with the estimate that 
\[\#\{(g_1,g_2,h_1,h_2)\in X_t^4:\, \dot{g_1}\dot{h_1}=\dot{g_2}\dot{h_2}\}\le |X_t|^3|E|.\]
Indeed, if we fix $g_1, g_2$, and $h_1$, then the number of $h_2$ is at most the size of $E$. 
Thus, it follows from Theorem \ref{thm_lowXY_G0} that
\[
|X_tX_t| \gg \min \left\{\frac{q^2|X_t|^4}{|X_t|^3|E|},\, \frac{|X_t|^2}{q^2}\right\} \gg \min\{q|E|,q^{-4}|E|^4\}\gg q^{\min\{1+\alpha, 4(\alpha-1)\}}\gg q^{\min\{2-\alpha,2\alpha-3\}}|X_t|.
\]
That is to say, we can take $\delta(\alpha)=\min\{2-\alpha,2\alpha-3\}>0$. The proof is completed.
\end{proof}

\section{Acknowledgements}
The authors would like to thank Junbin Dong for helpful discussions on the representation theory. The first author would like to thank the Vietnam Institute for Advanced Study in Mathematics (VIASM) for the hospitality and for the excellent working condition. The second author is supported by funds provided by ShanghaiTech University.

\end{document}